\documentclass[a4paper,11pt,leqno]{amsart}

\usepackage{amsthm}
\usepackage{amsmath}
\usepackage{amssymb}
\usepackage{mathrsfs}
\usepackage{latexsym}
\usepackage{exscale}
\usepackage{geometry}
\usepackage{graphicx}
\usepackage[utf8]{inputenc}
\usepackage{verbatim}
\usepackage{enumerate}
\usepackage{fancyhdr}

\usepackage{color}
\definecolor{red}{rgb}{1,0.1,0.1}
\definecolor{blue}{rgb}{0.1,0.1,1}
\definecolor{vb}{RGB}{160,32,240}

\theoremstyle{plain}

\newtheorem*{teo*}{Theorem}
\newtheorem*{prop*}{Proposition}
\newtheorem*{lema*}{Lemma}

\numberwithin{equation}{section}
\newtheorem{teo}{Theorem}[section]
\newtheorem{lema}[teo]{Lemma}

\newtheorem{prop}[teo]{Proposition}

\theoremstyle{remark}
\newtheorem{obs}[teo]{Remark}

\theoremstyle{definition}

\newtheorem*{mydef*}{Definition}

\newtheorem{ejem*}{Example}

\calclayout

\allowdisplaybreaks

\newcommand{\R}{\mathbb{R}^n}
\newcommand{\C}{\mathbb{C}}

\newcommand{\N}{\mathbb{N}}

\newcommand{\A}{\mathcal{A}}

\newcommand{\D}{\mathcal{D}}
\newcommand{\E}{\mathcal{E}}
\newcommand{\F}{\mathcal{F}}

\begin{document}
\title[Commutator of certain fractional type operators]{Commutators of certain fractional type operators with H\"ormander conditions, one-weighted and two-weighted inequalities}

\author[G.~H.~Iba\~{n}ez~Firnkorn]{Gonzalo H. Iba\~{n}ez-Firnkorn}
\address{G.~H.~Iba\~{n}ez~Firnkorn\\ FaMAF \\ Universidad Nacional de C\'ordoba \\
CIEM (CONICET) \\ 5000 C\'ordoba, Argentina}
\email{gibanez@famaf.unc.edu.ar}

\author[M.~S.~Riveros]{Mar\'{\i}a Silvina Riveros}
\address{M.~S.~Riveros \\ FaMAF \\ Universidad Nacional de C\'ordoba \\
CIEM (CONICET) \\ 5000 C\'ordoba, Argentina}
\email{sriveros@famaf.unc.edu.ar}

\thanks{ The authors are  partially supported by
CONICET and SECYT-UNC}




\begin{abstract}
 In this paper we study the commutators of fractional type
integral  operators. This operators are   given by kernels of the
form
$$K(x,y)=k_1(x-A_1y)k_2(x-A_2y)\dots k_m(x-A_my),$$
where $A_i$ are
invertibles matrices and  each $k_i$ satisfies a fractional size condition and
ge\-ne\-ra\-li\-zed fractional H\"ormander condition. We obtain weighted Coifman estimates, weighted
$L^p(w^p)$ - $L^q(w^q)$ estimates and  weighted BMO estimates. We
also give a
 two weight strong estimate for pair of weights of the form $(u,Su)$ where $u$ is an arbitrary non-negative function and  $S$ is a maximal operator
 depending  on the smoothness of the kernel $K$. For the
singular case  we also give a two-weighted endpoint estimate.
\end{abstract}

\maketitle


\section{Introduction}

In \cite{RS88},  Ricci and Sj\"ogren obtained the
$L^p(\mathbb{R},dx)$ boundedness, $p>1$, for a family of maximal
operators on the three dimensional Heisenberg group. Some of these
operators arise in the study of the boundary behavior of Poisson
integrals on the symmetric space $SL{\mathbb{R}^3}/SO(3)$. To get
the principal result, they studied the boundedness on
$L^2(\mathbb{R})$ of the operator
\begin{equation}\label{eq: Tinicial}
T_{\alpha}f(x)=\int_{\mathbb{R}} |x-y|^{-\alpha}|x+y|^{\alpha
-1}f(y)dy,
\end{equation}
for $0<\alpha<1$. Later, in \cite{GU93}, Godoy and Urciuolo studied a generalization of  (\ref{eq: Tinicial}) for $\R$.\\

During the last years, several authors studied operators such that
are generalizations of  (\ref{eq: Tinicial}) of the following form,
 let $0\leq \alpha<n$ and  $m \in \N$. For $1\leq i \leq m$, let $A_i$ be
matrices such that satisfy
\begin{equation*}
(H)\quad \qquad A_i \text{ is invertible and } A_i- A_j \text{ is invertible for }
i\not = j,1\leq i,j\leq m.
\end{equation*}
 For any  locally integrable bounded function $f$, $f \in L_{\text{loc}}^{\infty}(\R)$,  we define
\begin{equation}
T_{\alpha,m}f(x)=\int_{\R}K(x,y)f(y)dy, \label{eq: defT}
\end{equation}
where
\begin{equation} \label{eq: defK}
K(x,y)=k_1(x-A_1y)k_2(x-A_2y)\dots k_m(x-A_my).
\end{equation}
They consider particular functions $k_i$ and studied the operator in
different context: weighted Lebesgue and Hardy spaces with constant
and variable exponent, also the endpoint estimates and boundedness
in $BMO$ and weighted  $BMO$.
 See for example  \cite{FF15,GSU94,GU96,GU99,RU05,RiU13,R17,RoU13,U06,V16}.\\

 These operators  generalized classical operators as
$I_{\alpha}$, the fractional integral operator, and the rough
fractional and singular  operators. In several cases these type of
operators are not bounded in $H^p$, but instead are bounded from
$H^p$ into $L^q$, $0<p<1$ and some $q$ (see \cite{RU11,RU12}).
    In the case of $\alpha=0$, $T_{0,m}$ behaves   like a  singular integral operator.  If
$0<\alpha<n$,  $m=1$, $A_1= I$ and
$k_1(x-A_1y)=\frac1{|x-y|^{n-\alpha}}$ then
$T_{\alpha,1}=I_{\alpha}$.
\\

In \cite{RU05}, \cite{RiU13} and \cite{RU14}, Urciuolo and the second
author consider each $k_i$ as a rough fractional kernel, then each
$k_i$ satisfies a $L^{\alpha_i,r_i}$-H\"ormander regular condition, $k_i\in
H_{\alpha,r_i}$, that is, for all  $x\in \R$ and $R>|x|$
 \begin{align*} \sum_{m=1}^{\infty} (2^mR)^{n- \alpha} \| (K_{\alpha}(\cdot - x) - K_{\alpha}(\cdot))\chi_{B(x,2^{m+1}R)\setminus B(x,2^mR)}\|_{r_i,B(x,2^mR)} <\infty.\end{align*}

More recently, in \cite{IFR17}, we analized operators of the form
(\ref{eq: defT}) with
 conditions of regularity more generals that the
$L^{\alpha,r}$-H\"ormander condition and a fractional size condition . For the definitions of this conditions recall  that a function $\Psi :
[0,\infty) \rightarrow [0,\infty)$ is said to be a Young function if
$\Psi$ is continuous, convex, no decreasing and satisfies
$\Psi(0)=0$ and $\displaystyle \lim_{t \rightarrow \infty} \Psi(t)=
\infty$.

For each Young function $\Psi$ we can induce an average of the
Luxemburg norm of a function $f$ in the ball $B$ defined by
\begin{align*} \|f\|_{\Psi,B}:= \inf \left\{\lambda >0:\, \frac1{|B|} \int_{B} \Psi\left(\frac{|f|}{\lambda}\right) \leq 1  \right\},\end{align*}
and a fractional maximal operator $M_{\alpha,\Psi}$ defined by,
given $f \in L^1_{\text{loc}}(\R)$ and $0\leq \alpha<n$,
\begin{align*} M_{\alpha,\Psi}f(x) := {\underset{B \ni x}{\sup}} |B|^{\alpha/n}\|f\|_{\Psi,B}. \end{align*}

Now, we present the fractional size condition and a generalized
fractional H\"ormander condition. For more details  see \cite{BLR11}
or \cite{GIFR17}.

Let $\Psi$ be a Young function and let $0\leq \alpha < n$.
Let us introduce some notation: $|x|\sim s$ means $s < |x|\leq 2s$
and we write $$\|f\|_{ \Psi,|x|\sim s}=\|f \chi_{|x|\sim
s}\|_{\Psi,B(0,2s)}.$$

The function $K_{\alpha}$ is said to satisfies the fractional size condition, 
 if there exists a constant $C>0$ such that
$$\|K_{\alpha}\|_{\Psi,|x| \sim s} \leq C s^{\alpha - n}.$$
In this case we denote  $K_{\alpha}\in S_{\alpha,\Psi}$. When
$\Psi(t)=t$ we write $S_{\alpha,\Psi}=S_{\alpha}$. Observe that if
$K_{\alpha}\in S_{\alpha}$, then there exists a constant $c>0$ such
that
$$\int_{|x|\sim s} |K_{\alpha}(x)|dx\leq c s^{\alpha}.$$

The function  $K_{\alpha}$ satisfies the
$L^{\alpha,\Psi,k}$-H\"ormander condition ($K \in
H_{\alpha,\Psi,k}$), if there exist constants $c_{\Psi}>1$ and
$C_{\Psi}>0$ such that for all $x$ and $R>c_{\A}|x|$,
 \begin{align*} \sum_{m=1}^{\infty} (2^mR)^{n- \alpha} m^k \| K_{\alpha}(\cdot - x) - K_{\alpha}(\cdot)\|_{\Psi,|y|\sim2^mR} \leq C_{\Psi}.\end{align*}
We say that $K_{\alpha} \in H_{\alpha,\infty,k}$ if $K_{\alpha}$
satisfies the previous condition with $\|\cdot\|_{L^{\infty},|x|\sim
2^mR}$ in place of $\|\cdot\|_{\Psi,|x|\sim 2^mR}$. When $l=0$, we
write $H_{\alpha,\Psi}=H_{\alpha,\Psi,0}$.

When $\Psi(t)=t^r$, $1\leq r < \infty$, we simply write
$H_{\alpha,r,k}$ instead of $H_{\alpha,\Psi,k}$.
\\

In this paper, we study the $k$-order commutators of operator of the form (\ref{eq: defT})
where $k_i \in S_{n-\alpha_i, \Psi_i}\cap H_{n-\alpha_i,\Psi_i,k}$.\\
Recall that given  a locally integrable function $b$ and an operator
$T_\alpha$ defined as (\ref{eq: defT}), we define 
 the $k$-order commutator, $k\in \N\cup \{0\}$, by
 $$T_{\alpha, b}^k(f)=[b,T_{\alpha,b}^{k-1}]f=\int (b(x)-b(y))^kK(x,y)f(y)dy$$
  where we assume that $T_{\alpha,b}^0=T_\alpha$.

We also consider the following condition for the weights, there
exists $c>0$ such that
\begin{equation}
w(A_i x)\leq c w(x),\label{eq: condW}
\end{equation}
$a.e. x\in \R$ and for all $1\leq i \leq m$.

The following is an example of a weight $w$ that satisfies condition
(\ref{eq: condW}). Observe that also power weights satisfy  this
condition.
\begin{ejem*}
Let $ w(x)=
\begin{cases}
        \log\left(\frac1{|x|}\right)  & \text{if } |x| \leq \frac1{e} \\
        1 & \text{if } |x| > \frac1{e}
\end{cases}
$. Then  $w\in A_1$ and  satisfies $(\ref{eq: condW})$.
\end{ejem*}

The main results in this paper is the following Coifman type
estimate:
\begin{teo}\label{Coifman}
Let $b \in BMO$, $0\leq \alpha < n$,  $k\in \N\cup \{0\}$, $m \in
\N$ and $1\leq i \leq m$. Let $\Psi_i$ be Young functions and
$0\leq\alpha_i<n$ such that $\alpha_1+\cdots + \alpha_n=n-\alpha$.
Let $T_{\alpha,m}$ be the integral operator defined by (\ref{eq:
defT}) and $T_{\alpha,m,b}^k$ be the $k$-order commutator of
$T_{\alpha,m}$.
Suppose that the matrices $A_i$ satisfy the hypothesis $(H)$ and $k_i \in S_{n-\alpha_i, \Psi_i}\cap H_{n-\alpha_i,\Psi_i,k}$.\\
\noindent If $\alpha=0$,  let  $T_{0,m}$ be of strong type $(p_0,p_0)$ for some $1<p_0<\infty$.\\
Let $\varphi_k(t)=t\log(e+t)^k$ and let $\phi$ be  a Young function
such that
$\Psi_1^{-1}(t)\cdots\Psi_m^{-1}(t)\overline{\varphi_k}^{\,-1}(t)\phi^{-1}(t)\lesssim
t$ for $t\geq t_0$, some $t_0>0$.

 Let $0<p<\infty$. Then there exists $C>0$ such that, for $f\in
L_c^{\infty}(\R)$ and $w\in A_{\infty}$,
\begin{align}\label{Coifest}
\int_{\R} |T_{\alpha,m,b}^kf(x)|^pw(x)dx\leq C\|b\|_{BMO}^k \sum_{i=1}^m\int_{\R} |M_{\alpha,\phi}f(x)|^pw(A_ix)dx.
\end{align}
whenever the left-hand side is finite.\\
Furthermore, if $w\in A_{\infty}$ satisfying (\ref{eq: condW}), then
$$\int_{\R} |T_{\alpha,m,b}^kf(x)|^pw(x)dx\leq C\|b\|_{BMO}^{kp} \int_{\R} |M_{\alpha,\phi}f(x)|^pw(x)dx.$$
\end{teo}

To prove this estimate, we need a pointwise estimate that relates de
sharp delta maximal of the commutator with a sum of generalized
fractional maximal function of $f$.  As a consequence of the Coifman
estimate we get strong weighted estimates for the operator
$T_{\alpha,m,b}^k$ and  weighted BMO estimates. We also obtain
strong weighted estimates of the form
$$\|T_{\alpha,m,b}^kf\|_{L^p(u)}\leq c \|f\|_{L^p(Su)},$$
where $1<p<n/\alpha$, $u$ is any weight and where $S$ is appropriate
maximal operator. For $T_{0,m,b}^k$ we also give a two pair $(u,
Su)$ endpoint estimate. that is,
\begin{align} \label{debil2pesos}
u\{x\in \R : |T_{0,m,b}^k(x)|>\lambda\}\leq c\int_{\R}
\varphi_k\left(\frac{|f(x)|}{\lambda}\right)Su(x)dx.
\end{align}

The plan of the paper is the following, the next section contains
some preliminaries, definitions and previous results that are needed
to state the others results which appear in section 3. The proof of
the Coifman Theorem \ref{Coifman} is in the section 4.
 In section 5 we prove  strong one weighted inequalities and in section 6
 the two-weighted
inequalities.

\section{Preliminaries and previous results}
In this section we present some notions about Young function, Luxemburg norm and weights that will be fundamental throughout all this paper.  Also we present some previous results.\\

\subsection { Young Function and Luxemburg norm.}
Now, we present some extra definitions and properties for Young
functions. Also we given examples. For more details of these topics
see \cite{O65} or \cite{RaoRen91}.

Each Young function $\Psi$ has an associated complementary Young function $\overline{\Psi}$ satisfying
the generalized H\"older inequality
 \begin{align*} \frac1{|B|} \int_{B} |fg| \leq 2\|f\|_{\Psi,B}\|g\|_{\overline{\Psi},B}. \end{align*}

If $\Psi_1,\dots,\Psi_m,\phi$ are Young functions satisfying $\Psi_1^{-1}(t)\cdots\Psi_m^{-1}(t)\phi^{-1}(t)\leq c t$, for all $t \geq t_0$, some $t_0>0$ then
\begin{align}\label{Holder}
 \|f_1\cdots f_m g\|_{L^1,B} \leq c \|f_1\|_{\Psi_1,B}\cdots\|f_m\|_{\Psi_m,B}\|g\|_{\phi,B}, \end{align}
the function $\phi$ is called the complementary of the functions $\Psi_1,\dots,\Psi_m$.


Here are some examples of maximal operators related to certain Young
functions.
\begin{itemize}
\item $\Psi(t)=t$, then $\|f\|_{\Psi,Q}=f_Q:=\frac1{|Q|}\int_Q |f|$ and $M_{\alpha,\Psi}=M_{\alpha}$, the fractional maximal operator.

\item $\Psi(t)=t^{r}$ with $1< r<\infty$. In that case $\|f\|_{\Psi,Q}=\|f\|_{r,Q}:=\left(\frac1{|Q|}\int_Q |f|^r\right)^{1/r}$ and  
$M_{\alpha,\Psi}=M_{\alpha,r}$,  where $M_{0,r} f = M_rf  := M(f^r)^{1/r}$.

\item $\Psi(t)=\exp(t)-1$. Then, $M_{\alpha,\Psi}=M_{\alpha,\exp(L)}$.

\item If $\beta > 0$ and $1\leq r<\infty$, $\Psi(t)= t^r\log(e+t)^{\beta}$ is a Young function then $M_{\alpha,\Psi}=M_{\alpha,L^r(\log L)^{\beta}}$.
\item If $\alpha=0$ and $k\in\mathbb{N}$, $\Psi(t)= t\log(e+t)^{k}$ it can be proved
that $M_{\Psi}\approx M^{k+1}$, where $M^{k+1}$ is $M$ iterated $k+1$ times.
\end{itemize}

\begin{obs}\label{Malphar}
Observe that if $\Psi(t)=t^r$ then a simple computation shows that
$$M_{\alpha,r}f=\left(M_{\alpha r}|f|^r\right)^{1/r}.$$
\end{obs}

\begin{prop}\label{obsMA} Let $\D$ be a Young function and $A$ be a invertible matrix. Let $w_A(x)=w(Ax)$,  then
$$M_{\alpha,\D}(w_{A})(A^{-1}x) \leq c_{A,n} M_{\alpha,\D}(w)(x)$$
for almost every $x \in \R$.
\end{prop}
\begin{proof}
Fix $x\in \R$ and let $B=B(A^{-1}x,r)$ be a ball 
\begin{align*}
\frac1{|B|}\int_B
\D\left(\frac{w(Ay)}{\lambda}\right)dy=\frac1{|AB|}\int_{AB}
\D\left(\frac{w(z)}{\lambda}\right)dz.
\end{align*}
Then,  $x\in AB$ and
$$\|w_A\|_{\D,B}=\|w\|_{\D,AB}$$

Let $\|A\|_{\infty}=\sup_{x : |x|=1}|Ax|$.
There exist balls $B_1=B(x,\frac{r}{\|A^{-1}\|_{\infty}})$ and $B_2=B(x,\|A\|_{\infty}r)$ such that $B_1\subset AB \subset B_2$, then
$$\|w\|_{\D,AB}\leq \|A^{-1}\|_{\infty}^n\|A\|_{\infty}^n\|w\|_{\D,B_2}$$
Hence,
$$M^c_{\alpha,\D}(w_{A})(A^{-1}x)\leq \|A^{-1}\|_{\infty}^n\|A\|_{\infty}^n M^c_{\alpha,\D}w(x)$$
\end{proof}



\subsection {Weights.}  A weight
is a non negative locally integrable function in $\mathbb{R}^n$ that
takes values in $(0,\infty)$ almost every where.
 Let $0\leq \alpha<n$, $1\leq p,q \leq \infty$, we say that a weight
$w \in A_{p,q}$ if
$$[w]_{A_{p,q}}=\underset{B}{\sup}\|w\|_{q,B}\|w^{-1}\|_{p',B}<\infty,$$
 where the supremum is taken  over all balls $B\subset \R$.

If $1\leq p<\infty$, $A_p$ denotes the classical Muckenhoupt classes
of weights and $A_{\infty}=\cup_{p\geq 1} A_p$. Observe that $w\in
A_{p,p}$  if and only if $w^p\in A_p$ and $w\in A_{\infty,\infty}$
if,  and only if  $w^{-1}\in A_1$.

The fractional $B_p$ condition, $B_p^{\alpha}$, was introduced by
Cruz-Uribe and Moen in \cite{CUM13}:
 Let $1<p<n/\alpha$ and  $\frac1{q}=\frac1{p}-\frac{\alpha}{n}$.  A Young function $\phi \in B_p^{\alpha}$ if
$$\int_1^{\infty}\frac{\phi(t)^{q/p}}{t^q}\frac{dt}{t}<\infty.$$

They  proved that if $\phi \in B_p^{\alpha}$ then
$M_{\alpha,\phi}:L^p(dx)\rightarrow L^q(dx)$ and
$$\|M_{\alpha,\phi}\|_{L^p\rightarrow L^q}\leq c \left(\int_1^{\infty}\frac{\phi(t)^{q/p}}{t^q}\frac{dt}{t}\right)^{1/q}.$$

We will consider the following bump conditions: let  $1<q<\infty$ and
$\Psi$ be a Young function, then a weight $w\in A_{q,\Psi}$ if
$$[w]_{A_{q,\Psi}}=\underset{Q}{\sup}\|w\|_{q,Q}\|w^{-1}\|_{\Psi,Q} <\infty$$
 where the supremum is over all balls $B\subset \R$.

Let  $f$  be locally integrable function in $\mathbb{R}^n$. The
sharp maximal function is defined by
$$ M^{\#}f(x)=  {\underset{B\ni x}{\sup}}\frac1{|B|}\int_{B}\left|f(y)-\frac1{|B|}\int_{B}f(z)dz\right|dy.$$

A locally integrable function $f$ has bounded mean oscillation ($f
\in BMO$) if $ M^{\#}f \in L^{\infty}$ and the norm
$\|f\|_{BMO}=\|M^{\#}f\|_{\infty}$.

Observe that the $BMO$ norm is equivalent to
$$\|f\|_{BMO}=\|M^{\#}f\|_{\infty} \sim {\underset{B}{\sup}} {\underset{a \in \C}{\inf}} \frac1{|B|} \int_{B} |f(x)-a| dx. $$

There is also a weighted version of $BMO$, this is denoted by
$BMO(w)$, and it is described by the seminorm
$$\| |f| \|_w=\underset{B}{\sup}\|w\chi_B\|_{\infty}\left(\int_{B}\left|f(y)-\frac1{|B|}\int_{B}f(z)dz\right|dy\right).$$

It is easy to check that
$$\| |f| \|_w \simeq \|wM^{\#}f\|_{\infty}.$$

\subsection{ Previous results.}  Here we enounce some  known results
for  the operator $T_{\alpha, m}$. See \cite {IFR17}.

\begin{teo}\cite{IFR17}
Let $0\leq \alpha < n$, $m \in \N$ and let  $T_{\alpha,m}$ be the
integral operator defined by (\ref{eq: defT}). For $1\leq i \leq m$,
let $\Psi_i$ be Young functions,  $0\leq\alpha_i<n$ such that
$\alpha_1+\cdots + \alpha_m=n-\alpha$. Also suppose $k_i \in
S_{n-\alpha_i, \Psi_i}\cap H_{n-\alpha_i,\Psi_i}$ and let the matrices
$A_i$ satisfy  the hypothesis $(H)$.

\noindent If $\alpha=0$,  suppose $T_{0,m}$ is of strong type $(p_0,p_0)$ for some $1<p_0<\infty$.\\
If $\phi$ is the complementary of the functions $\Psi_1,\dots,\Psi_m$, then there exists $C>0
$ such that, for $0<\delta\leq 1$ and $f\in L_c^{\infty}(\R)$
\begin{equation}
M_{\delta}^{\sharp}|T_{\alpha,m}f|(x):= M^{\sharp}\left(|T_{\alpha,m}f|^{\delta}\right)(x)^{1/\delta} \leq C \sum_{i=1}^{m}M_{\alpha,\phi}f(A_i^{-1}x). \label{eq: Sharp2}
\end{equation}
\end{teo}

\begin{teo}\cite{IFR17}
Let $0\leq \alpha < n$ and $m \in \N$ and let  $T_{\alpha,m}$ be the
integral operator defined by (\ref{eq: defT}). For $1\leq i \leq m$,
let $\Psi_i$ be Young functions,  $0\leq\alpha_i<n$ such that
$\alpha_1+\cdots + \alpha_m=n-\alpha$. Also suppose $k_i \in
S_{n-\alpha_i, \Psi_i}\cap H_{n-\alpha_i,\Psi_i}$ and  that matrices
$A_i$ satisfy the hypothesis $(H)$.

\noindent  If $\alpha=0$,  suppose $T_{0,m}$ is of strong type $(p_0,p_0)$ for some $1<p_0<\infty$.\\
 Let $0<p<\infty$. If $\phi$ is the complementary of the functions
$\Psi_1,\dots,\Psi_m$, then there exists $C>0$ such that, for $f\in
L_c^{\infty}(\R)$ and $w\in A_{\infty}$,
$$\int_{\R} |T_{\alpha,m}f(x)|^pw(x)dx\leq C \sum_{i=1}^m\int_{\R} |M_{\alpha,\phi}f(x)|^pw(A_ix)dx,$$
whenever the left-hand side is finite.\\
\end{teo}

\section{Main results}
In this section we present the mains results
\subsection{Pointwise estimate} To obtain an appropriate
maximal operator which controls in weighted $L^p$ norms the operator
$T_{\alpha,m,b}^k$  we need the following result:
\begin{teo}\label{SharpHorm}
Let $b \in BMO$, $0\leq \alpha < n$,  $k\in \N\cup \{0\}$, $m \in
\N$ and $1\leq i \leq m$. Let $\Psi_i$ be Young functions and
$0\leq\alpha_i<n$ such that $\alpha_1+\cdots + \alpha_n=n-\alpha$.
Let $T_{\alpha,m}$ be the integral operator defined by (\ref{eq:
defT}) and $T_{\alpha,m,b}^k$ be the $k$-order commutator of
$T_{\alpha,m}$.
Suppose that the matrices $A_i$ satisfy the hypothesis $(H)$ and $k_i \in S_{n-\alpha_i, \Psi_i}\cap H_{n-\alpha_i,\Psi_i,k}$.\\
\noindent If $\alpha=0$,  let  $T_{0,m}$ be of strong type $(p_0,p_0)$ for some $1<p_0<\infty$.\\
Let $\varphi_k(t)=t\log(e+t)^k$ and let $\phi$ be  a Young function
such that
$\Psi_1^{-1}(t)\cdots\Psi_m^{-1}(t)\overline{\varphi_k}^{-1}(t)\phi^{-1}(t)\lesssim
t$ for $t\geq t_0$, some $t_0>0$.

Then, there exists $0<C=C(n,\alpha,A_1,...,A_m)$ such that, for
$0<\delta<\epsilon \leq 1$ and $f\in L_c^{\infty}(\R)$
\begin{equation}
M_{\delta}^{\sharp}|T_{\alpha,m,b}^kf|(x)\leq C \sum_{l=0}^{k-1}\|b\|_{BMO}^{k-l} M_{\epsilon}(T_{\alpha,m,b}^{l})+C\|b\|_{BMO}^k \sum_{i=1}^{m}
M_{\alpha,\phi}f(A_i^{-1}x). \label{eq: Sharp2}
\end{equation}
\end{teo}


\begin{table}[h!]
\begin{center}
\caption{Examples}
\vspace{-0.5cm}
\begin{tabular}{|r c|c|c|}
\hline &  $\Psi_i$\quad $1\leq i\leq m$ & $\phi$ & $M_{\alpha,\phi}$ \\
\hline
(i) & $\infty$ & $t\log(e+t)^k$  & $M_{\alpha,L\log L^k}$\\
(ii) & $t^{r_i},$ $1<r_i<\infty$ & $t^s\log(e+t)^{sk}$, $\displaystyle \sum_{i=1}^m \frac1{r_i}+\frac1{s}=1$ $$ & $M_{\alpha,L^s\log L^{sk}}$ \\ 
(iii) & $\psi_1=t^r$, $\psi_2(t)=\exp(t)-1$, $1<r<\infty$ & $t^{r'}\log(e+t)^{(k+1)r'}$ & $M_{\alpha,L^{r'}\log L^{r'(k+1)}}$ \\
\hline
\end{tabular}
\end{center}
\end{table}

This Theorem is a generalization of several known results. The  table ilustrate some example of this results: the example $(i)$ with $m=1$ is a classical example proved in \cite{BLR11}, $(ii)$ with $k=0$ is the example of fractional rough kernel proved in \cite{RU14} and the last example $(iii)$ is the commutator of the example given in \cite{IFR17}.

\subsection{One weight inequalities}

In this subsection, we prove the boundedness of the o\-pe\-ra\-tor,
$T_{\alpha,m,b}^k$  in two different ways, using the Coifman
inequality and  using a Cauchy integral formula. Also  we give a
weighted $BMO$ estimate for weights in the class $A(\frac{n}{\alpha
r},\infty)$.

\begin{teo}\label{TeoFuerte}
Let $0\leq \alpha < n$, $1<p<n/\alpha$ and
$\frac1{q}=\frac1{p}-\frac{\alpha}{n}$.  Let $T_{\alpha,m}$ be the
integral operator defined by (\ref{eq: defT}) under the hypothesis
of Theorem \ref{Coifman} and $T_{\alpha,m,b}^k$ be the $k$-order
commutator of $T_{\alpha,m}$.
 Suppose that one of the following hypothesis holds,

\begin{enumerate}[(a)]
\item If there exists $1<r<p$ such that ${\kappa}_r<\infty$. Let $\eta$ be a Young function such that $\eta^{-1}(t)t^{\frac{\alpha}{n}}\lesssim \phi^{-1}(t)$ for every $t>0$. If $\phi^{1+\frac{sn}{n-\alpha}} \in B_{\frac{sn}{n-\alpha}}$ for every $s>r(n-\alpha)/(n-\alpha r)$ and $w^r\in A(\frac{p}{r},\frac{q}{r})$,

\item Suppose that there exist $B$ and $C$ be Young functions such that $B^{-1}(t)C^{-1}(t)\leq \tilde{c} \phi^{-1}(t)$, $t>t_0>0$, $C\in B_p^{\alpha}$ and $w\in A_{q,B}$,

\item Suppose that the operator $T_{\alpha,m}$ is bounded from $L^p(w^p)$ into $L^q(w^q)$ for all $w\in A_{p,q}$.

\end{enumerate}

If $w$ satisfies the condition \eqref{eq: condW}  then there exists $c>0$ such that, for every $f\in L^p(w^p)$,
$$ \|T_{\alpha,m,b}^kf\|_{L^q(w^q)}\leq c \|b\|_{BMO}^{k}\|f\|_{L^p(w^p)}.$$

\end{teo}

 \begin{teo}\label{BMOw}
 Let $T_{\alpha,m}$ be
the integral operator defined by (\ref{eq: defT}) under the
hypothesis of Theorem \ref{Coifman} and $T_{\alpha,m,b}^k$ be the
$k$-order commutator of $T_{\alpha,m}$. Suppose there exists $r>1$
such that ${\kappa}_r<\infty$. If $w^r\in A(\frac{n}{\alpha
r},\infty)$ and satisfies (\ref{eq: condW}), then there exists $C>0$
such that for $f\in L_c^{\infty}(\R)$,
 $$\||T_{\alpha,m,b}^kf|\|_w\leq C\|b\|_{BMO}^k\|fw\|_{L^{n/\alpha}}.$$
 \end{teo}

\subsection{Two weights inequalities}

 If $T=T_{\alpha, m}$ is defined by (\ref{eq: defT}), then its
adjoint $T^{*}$ is
$$T^{*}g(x)=\int \tilde{k}_1(x -A_1^{-1}y)\cdots \tilde{k}_m(x -A_m^{-1}y)g(y)dy,$$
where $\tilde{k}_i(x)=k_i(-A_ix)$.  If $T^{*}$ satisfies hypothesis
of Theorem \ref{Coifman} then
$$
\int_{\R} |T^*f(x)|^qw(x)dx\leq c\int_{\R} \sum_{i=1}^m\left(M_{\alpha,\phi}f(A_i
x)\right)^qw(x)dx.
$$
for all $0<q<\infty$ and $w\in A_\infty$ 

\begin{teo}\label{fuerte2pesos}
Let $\phi$ be a Young function, $0\leq \alpha < n$ and $1<p<\infty$. Suppose that there exist Young functions $\E, \F$ such that $\E \in B_{p'}$ and $\E^{-1}(t)\F^{-1}(t)\leq \phi^{-1}(t)$. Set $\D(t)=\F(t^{1/p})$.

Let $T$ be a linear operator such that its adjoint $T^*$ satisfies
\begin{equation}\label{Coifman*}
\int_{\R} |T^*f(x)|^qw(x)dx\leq c\int_{\R} \sum_{i=1}^m\left(M_{\alpha,\phi}f(A_i
x)\right)^qw(x)dx,
\end{equation}
for all $0<q<\infty$ and $w\in A_{\infty}$.

 Then for any weight $u$,
\begin{align}\label{fuerteuMu}
\int_{\R}|Tf(x)|^pu(x)dx 
 \leq c\int_{\R} |f(x)|^p \sum_{i=1}^mM_{\alpha p,\D}u(A_i
x) dx.
\end{align}
\end{teo}

\begin{table}[h!]
\begin{center}
\caption{Examples}
\vspace{-0.3cm}
\begin{tabular}{|c|c|c|}
\hline
$\hspace{1cm}  M_{\alpha, \phi}  \hspace{1cm}     $ &  \hspace{1cm}       Range of $p$'s   \hspace{1cm}     &       \hspace{1cm} $M_{\alpha p,\D}$  \hspace{1cm} \\ \hline
  $M_{\alpha,L\log L^k}$ & $     1<p<\infty       $ & $M_{\alpha,L\log L^{(k+1)p-1+\varepsilon}}$\\
 $M_{\alpha,L^{r'}\log L^{r'(k+1)}}$ & $1<p<r$ &  $M_{\alpha,L^{\left(\frac{r}{p}\right)'}\log L^{\left(\frac{r}{p}\right)'((k+1)+p-1)+\varepsilon}}$\\
\hline
\end{tabular}
\end{center}
\end{table}

Now we give a  endpoint estimate for $T_{0,m,b}^k$ that derives
from Theorems \ref{Coifman} and \ref{fuerte2pesos}.

\begin{teo}\label{debilpesado} Let $T_{0 , m}$ be as in Theorem \ref{Coifman}.
\begin{enumerate}[(a)]
\item If there exists $r>1$ such that $t^r\leq c \phi(t)$ for $t\geq t_0>0$ then (\ref{debil2pesos}) holds for the pairs of weights $\displaystyle
\left(u,\sum_{i=1}^{m}\left(M_{\phi}u(A_i\cdot)\right)\right)$.
\item  If there exist  Young functions $\E, \F$ such that $\E \in B_{p'}$ and $\E^{-1}(t)\F^{-1}(t)\leq \phi^{-1}(t)$. Set
$\D(t)=\F(t^{1/p})$, then (\ref{debil2pesos}) holds for the pairs of
weights $\displaystyle
\left(u,\sum_{i=1}^{m}\left(M_{\D}u(A_i\cdot)\right)\right).$
\end{enumerate}
\end{teo}

\begin{obs}
 Observe that the pairs of weights given in $(a)$ are better
than the one in $(b)$. (see Remark 3.3 in \cite{LMPR09})
\end{obs}





\section{Proof of Sharp Theorem and  Coifman inequality}

Recall some  classical results concern to   functions in
$BMO$, we do not give  the proof.
\begin{lema}\label{BMO}
Let $b \in  BMO$.
\begin{enumerate}
\item For any measurable subsets $A\subset B \subset \R$ such that $|A|,|B|>0$,  we have
$$|b_A - b_B|\leq \frac{|B|}{|A|}\|b\|_{BMO}.$$
In particular,  if $\tilde{B}$ is a measurable set and $\tilde{B}_i=A_i^{-1}\tilde{B}$, $1\leq i \leq m$ then
$$|b_{\tilde{B}}-b_{(\cup_{l=1}^m \tilde{B}_l)\cup \tilde{B}}|\leq (1 + \sum_{l=1}^m |det(A_l^{-1})|) \|b\|_{BMO}$$
\item Let $B=B(c_B,R)$ be a ball, centered at $c_B$ with radius $R$, and $B^j=B(c_B,2^jR)$. Then,
$$|b_B - b_{B^j}|\leq c j \|b\|_{BMO}.$$
\end{enumerate}
\end{lema}

In the proof of Theorem \ref{SharpHorm}, we follow the idea of the
proof of Theorem 2.2 in \cite{RU14}.
\begin{proof}[Proof of Theorem \ref{SharpHorm}]
We just  consider the case $m=2$ and $k=1$, i. e.
$T_{\alpha,2,b}^1=[b,T_{\alpha,2}]$, and we will just write
$[b,T_{\alpha}]$. The general case is proved in an analogous way.

Let $f$ be a bounded function with compact support, $b \in BMO$ and
$0<\delta<\epsilon \leq 1$. Let $x\in \R$ and let $B=B(c_B,R)$ be a
ball that contains $x$, centered at $c_B$ with radius R.
We write $\tilde{B}=B(c_B,2R)$ and for $1\leq i \leq 2$, set $\tilde{B}_i=A_i^{-1}\tilde{B}$. Let $\displaystyle f_1=f\chi_{\cup_{i=1}^2 \tilde{B}_i}$ and $f_2=f-f_1$.\\
Suppose that $a:=T_{\alpha}((b-b_{\tilde{B} \cup \tilde{B}_1 \cup \tilde{B}_2})f_2)(c_B)<\infty$.

We can write
\begin{align*}
[b,T_{\alpha}f](x)
=(b(x)-b_{\tilde{B} \cup \tilde{B}_1 \cup \tilde{B}_2})T_{\alpha}f(x)-T_{\alpha}((b-b_{\tilde{B} \cup \tilde{B}_1 \cup \tilde{B}_2})f)(x).
\end{align*}
Now, we have
\begin{align}\label{delta}
&\left(\frac1{|B|}\int_B |[b,T_{\alpha}f](y)-a|^{\delta}dy\right)^{1/\delta} \nonumber\\
& \qquad\leq \left(\frac1{|B|}\int_B |(b(y)-b_{\tilde{B} \cup \tilde{B}_1 \cup \tilde{B}_2})T_{\alpha}f(y)|^{\delta}dy\right)^{1/\delta} \nonumber\\
& \qquad+\left(\frac1{|B|}\int_B |T_{\alpha}((b-b_{\tilde{B} \cup \tilde{B}_1 \cup \tilde{B}_2})f_1)(y)|^{\delta}dy\right)^{1/\delta} \nonumber\\
& \qquad+\left(\frac1{|B|}\int_B |T_{\alpha}((b-b_{\tilde{B} \cup \tilde{B}_1 \cup \tilde{B}_2})f_2)(y)-T_{\alpha}((b-b_{\tilde{B} \cup \tilde{B}_1 \cup \tilde{B}_2})f_2)(c_B)|^{\delta}dy\right)^{1/\delta} \nonumber\\
&\qquad=I+II+III.
\end{align}

To estimate $I$, let $q=\epsilon / \delta >1$, by H\"older's inequality and Lemma \ref{BMO},
\begin{align*}
I&\leq \left(\frac1{|B|}\int_B |(b(y)-b_{\tilde{B}})T_{\alpha}f(y)|^{\delta}dy\right)^{1/\delta}+|b_{\tilde{B}}-b_{\tilde{B} \cup \tilde{B}_1 \cup \tilde{B}_2}|\left(\frac1{|B|}\int_B |T_{\alpha}f(y)|^{\delta}dy\right)^{1/\delta}\\
&\leq \left(\frac1{|B|}\int_B |(b(y)-b_{\tilde{B}})|^{q'\delta}dy\right)^{1/q'\delta}\left(\frac1{|B|}\int_B |T_{\alpha}f(y)|^{q\delta}dy\right)^{1/q\delta}+C\|b\|_{BMO}M_{\delta}(T_{\alpha}f)(x)\\
&\leq C\|b\|_{BMO}M_{\epsilon}(T_{\alpha}f)(x)+C\|b\|_{BMO}M_{\delta}(T_{\alpha}f)(x)\\
&\leq C\|b\|_{BMO}M_{\epsilon}(T_{\alpha}f)(x).
\end{align*}

For $II$, by Jensen inequality
\begin{align}\label{local}
II&\leq \frac1{|B|}\int_B |T_{\alpha}((b-b_{\tilde{B} \cup \tilde{B}_1 \cup \tilde{B}_2})f_1)(y)|dy \nonumber \\
&\leq \frac1{|B|}\int_B \int_{\tilde{B}_1 \cup \tilde{B}_2} |K(y,z)||b(z)-b_{\tilde{B} \cup \tilde{B}_1 \cup \tilde{B}_2}||f_1(z)|dzdy \nonumber \\
&\leq \sum_{i=1}^2 \frac1{|B|}\int_{\tilde{B_i}}|b(z)-b_{\tilde{B}
\cup \tilde{B}_1 \cup \tilde{B}_2}||f_1(z)|\int_B  |K(y,z)|dy dz.
\end{align}
 We estimate the first summand, that is $z\in \tilde{B_1}$,
the case $z\in \tilde{B_2}$ is analogous. Observe that
\begin{equation}
\int_B  |K(y,z)|dy\leq \int_{\{y\in B: |y-A_1z|\leq|y-A_2z|\}}
|K(y,z)|dy + \int_{\{y\in B: |y-A_2z|\leq|y-A_1z|\}} |K(y,z)|dy.
\end{equation}

For 
 $j\in \N$, let us consider the set
$$C_j^1:= \{y\in B : |y-A_1z|\leq|y-A_2z|, |y-A_1z|\sim 2^{-j-1}R \}.$$
Observe that if $y\in B$ and $z\in \tilde{B_1}$ then $|y-A_1z|\leq 3R<4R$. \\
Thus,
\begin{align*}
&\int_{\{y\in B: |y-A_1z|\leq|y-A_2z|\}} |K(y,z)|dy \leq \sum_{j=-2}^{\infty} \int_{C_j^1}|K(y,z)|dy \\
&\quad \leq \sum_{j=-2}^{\infty} \frac{|A_1^{-1}B(c_B, 2^{-j}R)|}{|A_1^{-1}B(c_B, 2^{-j}R)|}\int_{A_1^{-1}B(c_B, 2^{-j}R)}|K(y,z)|\chi_{C_j^1}dy\\
&\quad \leq C \sum_{j=-2}^{\infty} |A_1^{-1}B(c_B,
2^{-j}R)|\|k_1(\cdot-A_1z)\|_{\Psi_1,|y-A_1z|\sim
2^{-j-1}R}\|k_2(\cdot-A_2z)\|_{\Psi_2,|y-A_1z|\sim 2^{-j-1}R}.
\end{align*}

Observe that if  $y\in C_j^1$ then $|y-A_2z|\geq|y-A_1z|>2^{j-1}R$
and since $k_2\in S_{n-\alpha_2,\Psi_2}$ we get
\begin{align}\label{tamK2}
\|k_2(\cdot-A_2z)\|_{\Psi_2,|y-A_1z|\sim 2^{-j-1}R}&\leq \sum_{k\geq 0} \|k_2(\cdot-A_2z)\|_{\Psi_2,|y-A_2z|\sim 2^{-j+k-1}R} \nonumber\\
&\leq \sum_{k\geq 0}  \|k_2(\cdot)\|_{\Psi_2,|y|\sim 2^{-j+k-1}R}\nonumber \\
&\leq \sum_{k\geq 0} (2^{-j+k}R)^{-\alpha_2}.
\end{align}

 As $k_1\in S_{n-\alpha_1,\Psi_1}$ and using inequality (\ref{tamK2})
 we get

\begin{align*}
\int_{\{y\in B: |y-A_1z|\leq|y-A_2z|\}} |K(y,z)|dy
 \leq C  \sum_{k\geq 0} (2^{-\alpha_2})^{k} \sum_{j=-2}^{\infty}
 (2^{-j}R)^{n-\alpha_1-\alpha_2}=CR^{\alpha}.
\end{align*}
In an analogous way, we get
\begin{equation}\label{eq: Klocal}
 \int_{\{y\in B: |y-A_2z|\leq|y-A_1z|\}} |K(y,z)|dy \leq CR^{\alpha}.
\end{equation}
Then, by (\ref{local}) and (\ref{eq: Klocal}), we obtain
\begin{align*}
II &\leq CR^{\alpha} \sum_{i=1}^2 \frac1{|B|}\int_{\tilde{B_i}}|b(z)-b_{\tilde{B} \cup \tilde{B}_1 \cup \tilde{B}_2}||f(z)|dz\\
&\leq CR^{\alpha} \sum_{i=1}^2 \frac1{|\tilde{B}_i|}\int_{\tilde{B}_i}(|b(z)-b_{\tilde{B}_i}|+|b_{\tilde{B}_i}-b_{\tilde{B} \cup \tilde{B}_1 \cup \tilde{B}_2}|)|f(z)|dz\\
&\leq C \sum_{i=1}^2  R^{\alpha} \left[\|b-b_{\tilde{B}_i}\|_{\exp L,\tilde{B}_i}\|f\|_{\phi,\tilde{B}_i}+\|b\|_{BMO}M_{\alpha}f(A_i^{-1}x)\right]\\
&\leq C \|b\|_{BMO} \sum_{i=1}^2 M_{\alpha,\phi}f(A_i^{-1}x).
\end{align*}

For $III$, by Jensen inequality we get
\begin{align*}
III&\leq \frac1{|B|}\int_B |T_{\alpha,2}((b-b_{\tilde{B} \cup \tilde{B}_1 \cup \tilde{B}_2})f_2)(y)-T_{\alpha,2}((b-b_{\tilde{B} \cup \tilde{B}_1 \cup \tilde{B}_2})f_2)(c_B)|dy\\
&\leq \frac1{|B|}\int_B \int_{(\tilde{B}_1 \cup \tilde{B}_2)^c}|K(y,z)-K(c_B,z)||b(z)-b_{\tilde{B} \cup \tilde{B}_1 \cup \tilde{B}_2}||f_2(z)|dzdy\\
&\leq \frac1{|B|}\int_B \int_{Z^l}|K(y,z)-K(c_B,z)||b(z)-b_{\tilde{B} \cup \tilde{B}_1 \cup \tilde{B}_2}||f_2(z)|dzdy,\\
\end{align*}
where
$$Z^l=(\tilde{B}_1 \cup \tilde{B}_2)^c \cap \{z: |c_B-A_lz|\leq |c_B-A_rz|, r\not = l, 1\leq r \leq 2\}.$$
Let us estimate $|K(y,z)-K(c_B,z)|$ for $y \in B$ and $z\in Z^l$,
\begin{align}\label{ineqK}
|K(y,z)-K(c_B,z)|\leq &|k_1(y-A_1z)-k_1(c_B-A_1z)||k_2(y-A_2z)| \nonumber \\
&+ |k_1(c_B-A_1z)||k_2(y-A_2z)-k_2(c_B-A_2z)|.
\end{align}

For simplicity we estimate the first summand of (\ref{ineqK}), the
other one follows in an analogous way. For $j\in \N$, let
$$D_j^l=\{z\in Z^l: |c_B-A_lz|\sim 2^{j+1}R\}.$$
Observe that $D_j^l\subset \{z:|c_B-A_lz|\sim 2^{j+1}R\}\subset
A_l^{-1}B(c_B,2^{j+2}R)=:\tilde{B}_{l,j}$. Using generalized
H\"older's inequality
\begin{align*}
\int_{Z^l}&|k_1(y-A_1z)-k_1(c_B-A_1z)||k_2(y-A_2z)||b(z)-b_{\tilde{B} \cup \tilde{B}_1 \cup \tilde{B}_2}||f(z)|dz \\
&\leq \sum_{j=1}^{\infty}\int_{D_j^l}|k_1(y-A_1z)-k_1(c_B-A_1z)||k_2(y-A_2z)||b(z)-b_{\tilde{B} \cup \tilde{B}_1 \cup \tilde{B}_2}||f(z)|dz\\
&\leq \sum_{j=1}^{\infty}\frac{|\tilde{B}_{l,j}|}{|\tilde{B}_{l,j}|}\int_{\tilde{B}_{l,j}}\Biggl[\chi_{\{z:|c_B-A_lz|\sim 2^{j+1}R\}}\chi_{D_j^l}|k_1(y-A_1z)-k_1(c_B-A_1z)||k_2(y-A_2z)| \\
&\hspace*{7.5cm} \left(|b(z)-b_{\tilde{B}_l^j}|+|b_{\tilde{B}_{l,j}}-b_{\tilde{B} \cup \tilde{B}_1 \cup \tilde{B}_2}|\right)|f(z)| \Biggr]dz\\
&\leq \sum_{j=1}^{\infty}|\tilde{B}_{l,j}|\|(k_1(y-A_1\cdot)-k_1(c_B-A_1\cdot))\chi_{D_j^l}\|_{\Psi_1,|c_B-A_lz|\sim 2^{j+1}R}\\
&\qquad \quad \|k_2(y-A_2\cdot)\chi_{Z^l}\|_{\Psi_2,|c_B-A_lz|\sim 2^{j+1}R} \Bigr( \|b-b_{\tilde{B}_l^j}\|_{\exp L,\tilde{B}_{l,j}} + c j \|b\|_{BMO} \Bigr)\|f_2\|_{\varphi,\tilde{B}_{l,j}}\\
&\leq c \|b\|_{BMO} \sum_{j=1}^{\infty}|\tilde{B}_{l,j}| j \|(k_1(y-A_1\cdot)-k_1(c_B-A_1\cdot))\chi_{D_j^l}\|_{\Psi_1,|c_B-A_lz|\sim 2^{j+1}R}\\
&\hspace*{7cm} \|k_2(y-A_2\cdot)\chi_{D_j^l}\|_{\Psi_2,|c_B-A_lz|\sim 2^{j+1}R}\|f_2\|_{\varphi,\tilde{B}_{l,j}}.
\end{align*}

Observe that $|c_B - A_lz|/2\leq |y-A_lz|<2|c_B - A_lz|$ and  if $|c_B - A_lz|\sim 2^{j+1}R$ then $2^jR\leq|y - A_lz|\leq 2^{j+2}R$. Thus, we have
\begin{align*}
\|k_l(y-A_l\cdot)\chi_{D_j^l}\|_{\Psi_l,|c_B-A_lz|\sim 2^{j+1}R}&\leq \|k_l(y-A_l\cdot)\|_{\Psi_l,|y-A_lz|\sim 2^{j}R} + \|k_l(y-A_l\cdot)\|_{\Psi_l,|y-A_lz|\sim 2^{j+1}R}\\
&\leq \|k_l(\cdot)\|_{\Psi_l,|x|\sim 2^jR} + \|k_l(\cdot)\|_{\Psi_l,|x|\sim 2^{j+1}R}\\
&\leq c (2^jR)^{-\alpha_l},
\end{align*}
where the last inequality holds since $k_l\in S_{n-\alpha_l,
\Psi_l}$. Also, by hypothesis
$$\|k_l(c_B-A_l\cdot)\chi_{D_j^l}\|_{\Psi_l,|c_B-A_lz|\sim 2^{j+1}R}\leq c(2^{j+1}R)^{-\alpha_l}.$$

For $r\not= l$, observe that if $z\in D_j^l$ then $|c_B-A_rz|\geq |c_B-A_lz|\geq 2^{j+1}R$, so we descompose $D_j^l=\cup_{k\geq j}(D_j^l)_{k,r}$ where
$$(D_j^l)_{k,r}=\{z\in D_j^l: |c_B-A_rz|\sim 2^{j+1}R\}.$$
Note that $(D_j^l)_{k,r}\subset\{z:|c_B-A_rz|\sim 2^{k+1}R\}$. Then,
as $k_r\in S_{n-\alpha_r, \Psi_r}$,
\begin{align*}
\|k_r(y-A_r\cdot)\chi_{D_j^l}\|_{\Psi_r,|c_B-A_lz|\sim 2^{j+1}R}&\leq \sum_{k\geq j} \|k_r(y-A_r\cdot)\chi_{(D_j^l)_{k,r}}\|_{\Psi_r,|c_B-A_lz|\sim 2^{j+1}R}\\
&\leq \sum_{k\geq j} \|k_r(y-A_r\cdot)\chi_{(D_j^l)_{k,r}}\|_{\Psi_r,|c_B-A_rz|\sim 2^{k+1}R}\\
&\leq \sum_{k\geq j} \|k_r(y-A_r\cdot)\|_{\Psi_r,|c_B-A_rz|\sim 2^{k+1}R}\\
&\leq  \sum_{k\geq j}\|k_r(\cdot)\|_{\Psi_r,|x|\sim 2^kR} + \|k_r(\cdot)\|_{\Psi_r,|x|\sim 2^{k+1}R}\\
&\leq c  \sum_{k\geq j}(2^kR)^{-\alpha_r} = c (2^jR)^{-\alpha_r}.
\end{align*}

Also, using again that $k_r\in S_{n-\alpha_r,
\Psi_r}$,, we get
\begin{align*}
\|k_r(c_B-A_r\cdot)\chi_{D_j^l}\|_{\Psi_r,|c_B-A_lz|\sim 2^{j+1}R}&\leq \sum_{k\geq j} \|k_r(c_B-A_r\cdot)\chi_{(D_j^l)_{k,r}}\|_{\Psi_r,|c_B-A_rz|\sim 2^{k+1}R}\\
&\leq c  \sum_{k\geq j}(2^{k+1}R)^{-\alpha_r} = c (2^jR)^{-\alpha_r}.
\end{align*}

Now for $l=1$,

\begin{align*}
\int_{Z^1}&|k_1(y-A_1z)-k_1(c_B-A_1z)||k_2(y-A_2z)||b(z)-b_{\tilde{B} \cup \tilde{B}_1 \cup \tilde{B}_2}||f_2(z)|dz \\
&\leq c \|b\|_{BMO} \sum_{j=1}^{\infty}(2^jR)^{n-\alpha_2} j \|(k_1(y-A_1\cdot)-k_1(c_B-A_1\cdot))\chi_{D_j^1}\|_{\Psi_1,|c_B-A_1z|\sim 2^{j+1}R}\|f\|_{\varphi,\tilde{B}_1^j}\\
&\leq c \|b\|_{BMO} M_{\alpha,\varphi}f(A_1^{-1}x)\sum_{j=1}^{\infty} (2^jR)^{n-\alpha_2-\alpha} j \|(k_1(y-A_1\cdot)-k_1(c_B-A_1\cdot))\chi_{D_j^1}\|_{\Psi_1,|c_B-A_1z|\sim 2^{j+1}R}\\
&\leq c \|b\|_{BMO} M_{\alpha,\varphi}f(A_1^{-1}x),
\end{align*}
where the last inequality follows since $k_1\in H_{n-\alpha_1,\Psi_1,1}$.

For $l=2$ we observe that
\begin{align*}
\|(k_1(y-A_1\cdot)-k_1(c_B-A_1\cdot))&\chi_{D_j^l}\|_{\Psi_1,|c_B-A_lz|\sim 2^{j+1}R}\\
&\leq \sum_{k\geq j}\|(k_1(y-A_1\cdot)-k_1(c_B-A_1\cdot))\chi_{(D_j^l)_{k,1}}\|_{\Psi_1,|c_B-A_1z|\sim 2^{k+1}R}.
\end{align*}
Then, we obtain
\begin{align*}
\sum_{j=1}^{\infty} (2^jR)^{\alpha_1} &j \|(k_1(y-A_1\cdot)-k_1(c_B-A_1\cdot))\chi_{D_j^1}\|_{\Psi_1,|c_B-A_1z|\sim 2^{j+1}R}\\
&\leq \sum_{j=1}^{\infty} (2^jR)^{\alpha_1} j \sum_{k\geq j}\|(k_1(y-A_1\cdot)-k_1(c_B-A_1\cdot))\chi_{(D_j^l)_{k,1}}\|_{\Psi_1,|c_B-A_1z|\sim 2^{k+1}R}\\
&\leq \sum_{j=1}^{\infty} \frac{(2^jR)^{\alpha_1}}{(2^kR)^{\alpha_1}} \sum_{k\geq j}(2^kR)^{\alpha_1}k\|(k_1(y-A_1\cdot)-k_1(c_B-A_1\cdot))\chi_{(D_j^l)_{k,1}}\|_{\Psi_1,|c_B-A_1z|\sim 2^{k+1}R}\\
&\leq \sum_{k=1}^{\infty} \left(\sum_{j=1}^k(2^{-\alpha_1})^{k-j}\right) (2^kR)^{\alpha_1}k\|(k_1(y-A_1\cdot)-k_1(c_B-A_1\cdot))\chi_{(D_j^l)_{k,1}}\|_{\Psi_1,|c_B-A_1z|\sim 2^{k+1}R}\\
&\leq c \sum_{k=1}^{\infty} (2^kR)^{\alpha_1}k\|(k_1(y-A_1\cdot)-k_1(c_B-A_1\cdot))\chi_{(D_j^l)_{k,1}}\|_{\Psi_1,|c_B-A_1z|\sim 2^{k+1}R}\\
&\leq c,
\end{align*}
where the last inequality follows since $k_1\in H_{n-\alpha_1,\Psi_1,1}$.

So as in the case $l=1$, we obtain
\begin{align*}
\int_{Z^l}&|k_1(y-A_1z)-k_1(c_B-A_1z)||k_2(y-A_2z)||b(z)-b_{\tilde{B} \cup \tilde{B}_1 \cup \tilde{B}_2}||f(z)|dz \\
&\leq c \|b\|_{BMO} M_{\alpha,\varphi}f(A_l^{-1}x).
\end{align*}

Then
\begin{align*}
III\leq c \|b\|_{BMO} \sum_{l=1}^2 M_{\alpha,\varphi}f(A_l^{-1}x).
\end{align*}

For the case $\alpha=0$, we repeat the same argument to the
inequality (\ref{delta}). The terms $I$ and $III$ are analogous to
the ones in  the case $0<\alpha<n$. For $II$, observe that  $T_0$ is
of weak-type $(1,1)$ with respect to the Lebesgue measure (see Lemma 5.3 in \cite{IFR17}), as $0<\delta<1$ and using Kolmogorov's inequality
(see Lemma 5.16 in \cite{Duo}) we get
$$II\leq \frac{C}{|B|}\int_{\R}|f_1(y)|dy=\sum_{i=1}^2  \frac{C}{|B|}\int_{\tilde{B_i}}|f_1(y)|dy\leq C \sum_{i=1}^2 Mf(A_i^{-1}f(x)),$$
and the theorem follows in this case.

\end{proof}

\begin{proof}[Proof of Theorem \ref{Coifman}]
By the extrapolation result Theorem 1.1 in \cite{CUMP04}, estimate (\ref{Coifest}) holds for all $0<p<\infty$ and all $w\in A_{\infty}$ if, and only if, it holds for some $0<p_0<\infty$ and all $w\in A_{\infty}$. Therefore, we will show  that  (\ref{Coifest}) is true for $p_0$, which is taken such that $\frac{n-\alpha}{n}<p_0<\infty$.
By homogeneity, we
assume that $\|b\|_{BMO}=1$. We proceed by induction.

 When $k=0$,
then $T_{\alpha,m,b}^0=T_{\alpha,m}$. As $k_i\in
H_{n-\alpha_i,\Psi_i,0}=H_{n-\alpha_i,\Psi_i}$, Theorem 3.3  in
\cite{IFR17} implies
$$\int_{\R} |T_{\alpha,m}f(x)|^pw(x)dx\leq C \sum_{i=1}^m\int_{\R} |M_{\alpha,\phi}f(x)|^pw(A_ix)dx.$$

Next, we assume that the results holds for all $0\leq j \leq k-1$
and let us see how to derive the case $k$.
 Let $w\in A_{\infty}$, then there
exists $r>1$ such that $w\in A_r$.
Let $0<\delta<1$ such that $1<r<p_0/\delta$, thus $w\in A_{p_0/\delta}$.
Then, by Lemma (5.1) in \cite{IFR17},  we have $\|T_{\alpha,m}f\|_{L^{p_0}(w)}<\infty$ and $\|(T_{\alpha,m}f)^{\delta}\|_{L^{p_0/\delta}(w)}<\infty$.

 %


For prove this, we consider $w,b\in L^{\infty}$,
\begin{align*}
\|T_{\alpha,m,b}^kf\|_{L^{p_0}(w)}&=\|\sum_{j=1}^k c_{k,j} b^{k-j}T_{\alpha,m}(b^jf)\|_{L^{p_0}(w)}
\leq \|w\|_{\infty}\|\sum_{j=1}^k c_{k,j} b^{k-j}T_{\alpha,m}(b^jf)\|_{L^{p_0}}<\infty,
\end{align*}
and $\|(T_{\alpha,m,b}^kf)^{\delta}\|_{L^{p_0/\delta}(w)}<\infty$.
Then, we get

\begin{align*}
\int_{\R}|T_{\alpha,m,b}^kf(x)|^{p_0}w(x)dx&\leq \int_{\R}|M(T_{\alpha,m,b}^kf)^{\delta}(x)|^{p_0/\delta}w(x)dx
\\&\leq \int_{\R}(M_{\delta}^{\sharp}(T_{\alpha,m,b}^kf)(x))^{p_0}w(x)dx
\\&\leq C \sum_{l=0}^{k-1} \|(M_{\epsilon}(T_{\alpha,m,b}^{l}f)\|_{L^{p_0}(w)}^{p_0} + C \sum_{i=1}^{m} \int_{\R}(M_{\alpha,\phi}f(A_i^{-1}x))^{p_0}w(x)dx.
\end{align*}

Since $\delta<q/r<1$, we can take $\epsilon>0$ such that $\delta<\epsilon<{p_0}/r<1$, and so $w\in A_{p_0/\epsilon}$. Hence,
$$\|(M_{\epsilon}(T_{\alpha,m,b}^{l}f)\|_{L^{p_0}(w)}=\|(M(|T_{\alpha,m,b}^{l}f|^{\epsilon})\|_{L^{p_0/\epsilon}(w)}^{1/\epsilon}\leq c\|T_{\alpha,m,b}^{l}f\|_{L^{p_0}(w)}.$$

Thus, the induction hypothesis implies that, for any $0\leq l \leq k-1$,
$$\|(M_{\epsilon}(T_{\alpha,m,b}^{l}f)\|_{L^{p_0}(w)}^{p_0}\leq c\|T_{\alpha,m,b}^{l}f\|_{L^{p_0}(w)}^{p_0}\leq c  \sum_{i=1}^{m} \int_{\R}(M_{\alpha,\phi}f(A_i^{-1}x))^{{p_0}}w(x)dx.$$

Hence, for $w$ and $ b\in L^{\infty}$, \eqref{Coifest} holds, that is
\begin{equation*}\label{eq: Coif}
\int_{\R}|T_{\alpha,m,b}^kf(x)|^{p_0}w(x)dx\leq  C \sum_{i=1}^{m}
\int_{\R}(M_{\alpha,\phi}f(x))^{{p_0}}w(A_ix)dx.
\end{equation*}

For the general case, if $b\in BMO$, for any $N\in \N$ we define $b_N=b\chi_{[-N,N]}+ N\chi_{(N,\infty)} - N\chi_{(-\infty,-N)}$,
 then $\|b_N\|_{\infty}=\|b_N\|_{BMO}\leq 2\|b\|_{BMO}$. For the weight $w\in A_{\infty}$,
 we define $w_N=\min \{w,N\}$, then $w_N\in A_{\infty}$ and $[w_N]_{A_\infty}\leq [w]_{A_{\infty}}$.
 Now, using convergence theorems, for details see \cite{LMRT08}, we conclude  that (\ref{Coifest}) holds for any $b\in BMO$ and $w\in A_{\infty}$.

 Thus, as mentioned, using the extrapolation results obtained in \cite{CUMP04}, (\ref{Coifest}) holds for all $0< p <\infty$,  $b\in BMO$ and $w\in A_\infty$.

If $w$ satisfies (\ref{eq: condW}), we have
\begin{align*}
\int_{\R}|T_{\alpha,m,b}^kf(x)|^{p}w(x)dx &\leq  C \sum_{i=1}^{m} \int_{\R}(M_{\alpha,\phi}f(x))^{{p}}w(A_ix)dx
\\&\leq C \sum_{i=1}^{m} \int_{\R}(M_{\alpha,\phi}f(x))^{{p}}w(x)dx.
\end{align*}
\end{proof}

\section{Proof of one weighted inequalities}

For the proof of Theorem \ref{TeoFuerte} $a)$ and $b)$, we need
the Coifman inequality \eqref{Coifest} and the boundedness of the
maximal operator,  given in \cite{BDP14} (see Theorem 2.6 ). In the case of the
classical Lebesgue spaces the theorem is the following

\begin{teo}\cite{BDP14}
Let $0\leq \alpha <n$, $w$ be a weight, $1\leq \beta< p<n/\alpha$ and
$1/q=1/p-\alpha/n$. Let $\eta$ be a Young function such that $\eta^{1+\frac{\rho \alpha}{n-\alpha}} \in B_{\frac{\rho n}{n-\alpha}}$ for every $\rho >\beta(n- \alpha)/(n -\alpha\beta)$, and let $\phi$ be a Young funciton such that $\phi^{-1}(t)t^{\alpha /n} \lesssim \eta^{-1}(t)$ for every $t>0$.  If  $w ^{\beta}\in A_{(\frac{p}{\beta},\frac{q}{\beta})}$, then $M_{\alpha,\eta}$ is bounded form $L^p(w^p)$ into $L^q(w^q)$.
\end{teo}
The boundedness of the $M_{\alpha,\phi}$ from $L^{p}(w^p)$ into
$L^q(w^q)$ with bump conditions,  given in
\cite{LibroPesos} (see Theorem 5.37), is the following,

\begin{teo}\cite{LibroPesos}
Let $0\leq\alpha<n$, $1<p<n/\alpha$, let
$\frac1{q}=\frac1{p}-\frac{\alpha}{n}$. Let $\phi, B$ and $C$ be
Young functions such that $B^{-1}(t)C^{-1}(t)\leq c \phi^{-1}(t)$,
$t\geq t_0 > 0$. If $C\in B_p^{\alpha}$ and  $w\in A_{q,B}$, then
for every $f\in L^p(w^p)$,
$$\int (M_{\alpha,\phi}f)^qw^q\leq C \int |f|^p w^p.$$
\end{teo}

Now we prove the part $(a)$ and $(b)$ of Theorem \ref{TeoFuerte},
\begin{proof}[Proof of Theorem \ref{TeoFuerte} a) and b)]
From the previous Theorems,  hypothesis (a) or (b) implies that $M_{\alpha,\phi}$ is bounded from $L^p(w^p)$ into $L^q(w^q)$.\\
 Then, by Theorem \ref{Coifman} and $w$ satisfies \eqref{eq: condW},
$$ \|T_{\alpha,m,b}^kf\|_{L^q(w^q)}\leq c \|b\|_{BMO}^{k} \|M_{\alpha,\phi}f\|_{L^q(w^q)} \leq c \|b\|_{BMO}^{k} \|f\|_{L^p(w^p)}.$$
\end{proof}

For the proof of Theorem 4.1 $(c)$    we use a Cauchy integral
formula technique, see \cite{CPP12} and \cite{BMMST17}. This
technique is as follows, let $T$ be a linear operator, we can write
$T_b^k$ as a complex integral operator
$$T_b^kf=\frac{d^k}{dz^k} e^{zb}T(f e^{-zb})\bigg|_{z=0} = \frac1{2\pi i}\int_{|z|=\epsilon}\frac{T_z(f)}{z^2}dz,$$
where $\epsilon >0$ and $T_z(f)=e^{zb}T(f e^{-zb})$, $z\in \C$.
This is called the ``conjugation'' of $T$ by $e^{zb}$. Now, if $\|\cdot\|$ is a norm we can apply Minkowski inequality,
$$\|T_b^kf\|\leq \frac1{2\pi \epsilon^k} \underset{|z|=\epsilon}{\sup} \|T_z(f)\| \qquad \epsilon >0.$$

Observe that using this technique we can obtain the boundedness of the commutator using the boundedness of the conjugation of the operator.


\begin{lema}\cite{BMMST17}\label{expAp}
Fix $1<r,\eta<\infty$. If $w^{\eta}\in A_r$ and $b\in BMO$.
Then $we^{\lambda b}\in A_r$ for every $\lambda \in \mathbb{R}$ verifying
$$|\lambda|\leq \frac{\min\{1,p-1\}}{\eta'\|b\|_{BMO}}.$$
\end{lema}

\begin{proof}[Proof of  Theorem \ref{TeoFuerte} (c)]
Let $T=T_{\alpha,m}$.
Let $w\in A_{p,q}$ and $\nu = w e^{Re(z)b}$, where $Re(z)$ is the real part of the complex number $z$.
If $\nu \in A_{p,q}$, then
$$\|T_zf\|_{L^q(w^q)}=\|T(f e^{-zb})\|_{L^q(\nu^q)}\leq c\|f e^{-zb}\|_{L^p(\nu^p)}=c\|f\|_{L^p(w^p)},$$
since $T$ is boundedness from $L^p(\nu^p)$ into $L^q(\nu^q)$.

Let us prove that $\nu \in A_{p,q}$. If $w\in A_{p,q}$ then $w^q\in A_{1+\frac{q}{p'}}$ and exists $r>1$ such that $w^{qr}\in A_{1+\frac{q}{p'}}$.
 Let $\epsilon_0=\frac{\min\{1,\frac{p'}{q}\}}{qr'\|b\|_{BMO}}$, if $|z|=\epsilon_0$ then
$$|qRe(z)|\leq q|z|=\frac{\min\{1,\frac{p'}{q}\}}{r'\|b\|_{BMO}}.$$
By Lemma \ref{expAp}, $\nu^q\in A_{1+\frac{q}{p'}}$ and $\nu\in A_{p,q}$.



Hence,
\begin{align*}
\|T_{b}^kf\|_{L^p(w^p)} &\leq \frac1{2\pi \epsilon_0^k} \underset{|z|=\epsilon_0}{\sup} \|T_z(f)\|_{L^p(w^p)} \\
&\leq \frac1{2\pi c_{p,q}^k}\|b\|_{BMO}^k \|f\|_{L^q(w^q)}.
\end{align*}
\end{proof}

Now, we prove the weighted $BMO$ inequality
\begin{proof}[Proof of Theorem \ref{BMOw}] Follows the ideas in \cite{LMRT08}, the authors prove that
\begin{equation}\label{eq: MBMO}
w^r\in A\left(\frac{n}{\alpha r},\infty\right) \Rightarrow \|wM_{\alpha,r}f\|_{\infty}\leq C \|fw\|_{n/\alpha}.
\end{equation}

Now, by Lemma 4.1 in \cite{CLO17}, Theorem \ref{SharpHorm} and (\ref{eq: MBMO}), we get
\begin{align*}
\||T_{\alpha,m,b}^kf|\|_w \simeq \|wM^{\sharp}T_{\alpha,m,b}^kf\|_{\infty}&\leq C \|b\|_{BMO}^k\sum_{i=1}^m \|wM_{\alpha,\phi}f(A_i^{-1} \cdot)\|_{\infty}
\\&\leq C \kappa_r\|b\|_{BMO}^k\sum_{i=1}^m \|wM_{\alpha,r}f(A_i^{-1} \cdot)\|_{\infty}
\\&\leq C \kappa_r\|b\|_{BMO}^k\sum_{i=1}^m \|wf(A_i^{-1} \cdot)\|_{n/\alpha}
\\&\leq C \kappa_r\|b\|_{BMO}^k\sum_{i=1}^m \|w(A_i \cdot)f\|_{n/\alpha}
\\&\leq C \kappa_r m \|b\|_{BMO}^k \|wf\|_{n/\alpha}.
\end{align*}
\end{proof}

\section{Proof of two weights norm inequalities}
%
%


 For the proof of the two weights norm inequality we need the following auxiliary results.
\begin{lema}\label{lemaM}

\begin{enumerate}[(a)]
\item \cite{P95sufficient} Let $\Phi$  be a Young function. If $\Phi \in B_p$ then for every weight $\nu$
$$\int |M_{\Phi}f(x)|^p\nu(x)dx\leq c \int |f(x)|^p M\nu(x) dx.$$
\item \cite{LMPR09} If $r>1$, then
$$M(M_r)\approx M_r.$$
\end{enumerate}
\end{lema}

\begin{proof}[Proof of Theorem \ref{fuerte2pesos}]
Let $u$ a weight and $\nu(x)= M_{\alpha p,\D}u(x)$.
By duality, (\ref{fuerteuMu}) turns out to be equivalent to
$$\int_{\R}|T^*f(x)|^{p'}  \nu(x)^{1-p'}dx
 \leq c\int_{\R} \sum_{i=1}^m |f(A_ix)|^{p'} u(x)^{1-p'}dx.$$

 Since $\nu=M_{\alpha p,\D}u^{1-p'}\in A_{\infty}$, see \cite{BLR11}, 
then by Remark \ref{obsMA} and the fact that $\E \in B_{p'}$ we get
\begin{align*}
\int_{\R}|T^*f(x)|^{p'}  \nu(x)^{1-p'}dx
&\leq c \int_{\R}M_{\alpha,\phi}f(A_ix)^{p'}  \nu(x)^{1-p'}dx
\\ &\leq c \int_{\R}M_{\E}(fw_{A_i^{-1}}^{-1/p})(A_ix)^{p'}M_{\alpha,\F}(w_{A_i^{-1}}^{1/p})(A_ix)^{p'}  \nu(x)^{1-p'}dx
\\ &= c \int_{\R}M_{\E}(fw_{A_i^{-1}}^{-1/p})(A_ix)^{p'}M_{\alpha p,\D}(w_{A_i^{-1}})(A_ix)^{p'/p}  \nu(x)^{1-p'}dx
\\ &\leq c \int_{\R}M_{\E}(fw_{A_i^{-1}}^{-1/p})(A_ix)^{p'}M_{\alpha p,\D}(w)(x)^{p'/p}  \nu(x)^{1-p'}dx
\\ &\leq c \int_{\R}M_{\E}(fw_{A_i^{-1}}^{-1/p})(A_ix)^{p'}dx
\\  &\leq c \int_{\R}|f(A_ix)w_{A_i^{-1}}^{-1/p}(A_ix)|^{p'}dx=c \int_{\R}|f(A_ix)|^{p'}w(x)^{1-p'}dx.
\end{align*}

\end{proof}

\begin{proof}[Proof of Theorem \ref{debilpesado}]%
We consider $m=2$, $T=T_{0,2}$. The general case is analogous.\\
Let $u$ be a weight, suppose that $u\in L^{\infty}_c$ (otherwise consider $u_N=\min \{u,N\}\chi_{B(0,N)}$ and use monotone converge). Let $0\leq f\in L^{\infty}_c$. By the standard Calder\'on-Zygmund decomposition of $f$ at  height $\lambda$, then there exists $\{Q_j\}_j$ dyadic cubes such that
$$\lambda < \frac1{|Q_j|}\int_{Q_j}f\leq 2^n \lambda,$$
and write $f=g+h$ where
\begin{align*}
 g=f\chi_{\R\setminus \cup_j Q_j} + \sum_j f_{Q_j}\chi_{Q_j}, && h=\sum_j h_j =\sum_j (f-f_{Q_j})\chi_{Q_j},
 \end{align*}
where $f_{Q_j}$ denotes the average of $f$ over $Q_j$. Let us recall that $0\leq g \leq 2^n\lambda$ a.e. and also that each $h_j$ has vanishing integral. We set $\tilde{Q}_{j,i}$ the cube with center $A_ic_j$ with length $2\sqrt{n}Ml(Q_j)$, where $M=\underset{1\leq i \leq 2}{\max}\|A_i\|_{\infty}$, $\displaystyle \tilde{\Omega}=\bigcup_j \left( \tilde{Q}_{j,1}\cup \tilde{Q}_{j,2}\right)$ and $\tilde{u}=u\chi_{\R\setminus \tilde{\Omega}}$. Then
\begin{align*}
u\{x\in \R : |Tf(x)|> \lambda\} &\leq u(\tilde{\Omega}) + u\{x \in  \R\setminus \tilde{\Omega}: |Th(x)|>\lambda/2\}
\\&\qquad \qquad \qquad+u\{x \in  \R\setminus \tilde{\Omega}: |Tg(x)|>\lambda/2\}
\\ &\quad= I + II + III.
\end{align*}

For $I$, observe that $|\tilde{Q}_{j,i}|=(42\sqrt{n}M)^n |Q_j|$. Then, we have

\begin{align*}
I&=  u(\bigcup_j \left( \tilde{Q}_{j,1}\cup \tilde{Q}_{j,2}\right))\leq \sum_j \bigg[ u(\tilde{Q}_{j,1})+ u(\tilde{Q}_{j,2}) \bigg]
\\&\leq \frac{c_n}{\lambda}\sum_j \Bigg[\frac{u(\tilde{Q}_{j,1})}{|\tilde{Q}_{j,1}|}+ \frac{u(\tilde{Q}_{j,2})}{|\tilde{Q}_{j,2}|}\Bigg] \int_{Q_j} f
\\&\leq \frac{c_n}{\lambda}\sum_j  \int_{Q_j} \big[Mu(A_1x) + Mu(A_2x)\big] f(x)dx.
\end{align*}
where the last inequality follows since $x\in Q_j$ then $A_ix\in \tilde{Q}_{j,i}$.

To estimate $II$, recall that  the function $h_j$ has vanishing integral, then
\begin{align*}
II&=u\{x \in  \R\setminus \tilde{\Omega}: |Th(x)|>\lambda/2\} \leq \frac{2}{\lambda}\sum_j \int_{\R\setminus \tilde{\Omega}}|Th_j(x)|u(x)dx
\\ &\leq \frac{2}{\lambda}\sum_j \int_{\R\setminus \tilde{\Omega}}\left|\int_{Q_j}(K(x,y)-K(x,c_j))h_j(y)dy\right|u(x)dx
\\ &\leq \frac{2}{\lambda}\sum_j \int_{Q_j} |h_j(y)|\int_{\R\setminus \left( \tilde{Q}_{j,1}\cup \tilde{Q}_{j,2}\right)}\left|(K(x,y)-K(x,c_j))\right|u(x)dxdy.
\end{align*}

We claim that for every $y\in Q_j$ we have
\begin{align}\label{essinf}
\int_{\R\setminus \left( \tilde{Q}_{j,1}\cup \tilde{Q}_{j,2}\right)}\left|(K(x,y)-K(x,c_j))\right|u(x)dx
\leq c \underset{x\in Q_j}{ \text{ ess\,}\inf}\left[M_{\Phi}u(A_1x)+M_{\Phi}u(A_2x)\right].
\end{align}
This estimate drives us to

\begin{align*}
II &\leq \frac{c}{\lambda}\sum_j \underset{ Q_j}{ \text{ ess\,}\inf}\left[M_{\Phi}u(A_1\cdot)+M_{\Phi}u(A_2\cdot)\right] \int_{Q_j} |h_j(y)| dy
\\&\leq \frac{c}{\lambda}\sum_j \underset{ Q_j}{ \text{ ess\,}\inf}\left[M_{\Phi}u(A_1\cdot)+M_{\Phi}u(A_2\cdot)\right] \int_{Q_j} f(y) dy
\\ &\leq \frac{c}{\lambda}\sum_j \int_{Q_j} f(y) \left[M_{\Phi}u(A_1y)+M_{\Phi}u(A_2y)\right] dy.
\end{align*}

Let us proof (\ref{essinf}). Using (\ref{ineqK}), we obtain
\begin{align*}
\int_{\R\setminus \left( \tilde{Q}_{j,1}\cup \tilde{Q}_{j,2}\right)}&\left|(K(x,y)-K(x,c_j))\right|u(x)dx
\\&\leq \int_{Z^1\cup Z^2}|k_1(x-A_1y)-k_1(x-A_1c_j)||k_2(x-A_2y)|u(x)dx
\\&\quad+\int_{Z^1\cup Z^2}|k_1(x-A_1c_j)||k_2(x-A_2y)-k_2(x-A_2c_j)|u(x)dx,
\end{align*}
where $Z^i=\R\setminus \left( \tilde{Q}_{j,1}\cup \tilde{Q}_{j,2}\right)\cap \{x: |x-A_iy|\leq |x-A_ry|, r\not=i\}$.

We only estimate the first summand, the other follows in an analogous way. Using generalized H\"older inequality and observing that $|\tilde{Q}_{j,i}|=(42\sqrt{n}M)^n |Q_j|$, we have
\begin{align*}
\int_{Z^1}&|k_1(x-A_1y)-k_1(x-A_1c_j)||k_2(x-A_2y)|u(x)dx
\\ &\leq c \sum_{t=1}^{\infty}  |Q^t| \|k_1(\cdot-A_1y)-k_1(\cdot-A_1c_j)\chi_{Q^{t+1}\setminus Q^t}\|_{\Psi_1, Q^{t+1}} \|k_2(\cdot-A_2y)\chi_{Q^{t+1}\setminus Q^t}\|_{\Psi_2,Q^{t+1}} \|u\|_{\Phi,Q^{t+1}},
\end{align*}
where $Q^t$ is the cube with center $A_1c_j$ and lenght $2^t\sqrt{n}Ml(Q_j)$. Observe $Q^1=\tilde{Q}_{j,1}$.

Since, $k_2\in S_{n-\alpha_2,\Phi_2}$, we get
$$\|k_2(\cdot-A_2y)\chi_{Q^{t+1}\setminus Q^t}\|_{\Psi_2,Q^{t+1}}\leq c |Q^t|^{-\alpha_2/n}.$$
Also, if $x\in Q_j$ then for all $t\in \N$ we get   $A_1x\in \tilde{Q}_{j,1}\subset Q^{t}$  and
$$|Q^t|^{\frac{\alpha}{n}}\|u\|_{\Phi,Q^{t+1}}\leq c \underset{Q_j}{ \text{ ess\,}\inf}M_{\Phi}u(A_1\cdot).$$

Then,
\begin{align*}
\int_{Z^1}&|k_1(x-A_1y)-k_1(x-A_1c_j)||k_2(x-A_2y)|u(x)dx
\\ &\leq c  \underset{Q_j}{ \text{ ess\,}\inf}M_{\Phi}u(A_1\cdot)\sum_{t=1}^{\infty}  |Q^t|^{\frac{\alpha_1}{n}} \|k_1(\cdot-A_1y)-k_1(\cdot-A_1c_j)\chi_{Q^{t+1}\setminus Q^t}\|_{\Psi_1, Q^{t+1}}
\\ &\leq c  \underset{Q_j}{ \text{ ess\,}\inf}M_{\Phi}u(A_1\cdot),
\end{align*}
where the last inequality holds since $k_1\in H_{n-\alpha_1,\Psi_1}$.

In an analogous way, we obtain
\begin{align*}
\int_{Z^2}|k_1(x-A_1y)-k_1(x-A_1c_j)||k_2(x-A_2y)|u(x)dx
\leq c  \underset{Q_j}{ \text{ ess\,}\inf}M_{\Phi}u(A_2\cdot).
\end{align*}

The estimate  $III$ is different in each case. We start with $(a)$. For $p>1$, using Theorem \ref{Coifman}, the fact that  $M_ru\in A_1$ and Lemma \ref{lemaM}, we get
\begin{align*}
III&=u\{x \in  \R\setminus \tilde{\Omega}: |Tg(x)|>\lambda/2\}\leq \frac{2^p}{\lambda^p}\int_{\R} |Tg(x)|^p \tilde{u}(x)dx
\\&\leq \frac{2^p}{\lambda^p}\int_{\R} |Tg(x)|^p M_r\tilde{u}(x)dx
\leq \frac{c}{\lambda^p}\sum_{i=1}^2\int_{\R} |M_{\Phi}g(A_i^{-1}x)|^p M_r\tilde{u}(x)dx
\\ &\leq \frac{c}{\lambda^p}\sum_{i=1}^2\int_{\R} |g(A_i^{-1}x)|^p M(M_r\tilde{u})(x)dx
\leq \frac{c}{\lambda^p}\sum_{i=1}^2\int_{\R} |g(A_i^{-1}x)|^p M_r\tilde{u}(x)dx
\\&\leq \frac{c}{\lambda^p}\int_{\R} |g(x)|^p \sum_{i=1}^2M_r\tilde{u}(A_ix)dx
\leq \frac{c}{\lambda^p}\int_{\R} |g(x)|^p \sum_{i=1}^2M_{\Phi}\tilde{u}(A_ix)dx,
\end{align*}
where the last inequality holds since $t^r\leq \Phi(t)$ for $t\geq t_0 >0$.
Since $g\leq 2^n\lambda$, then
\begin{align*}
II&\leq \frac{c}{\lambda^p}\int_{\R} |g(x)|^p \sum_{i=1}^2M_{\Phi}\tilde{u}(A_ix)dx
\\ &\leq  \frac{c}{\lambda}\int_{\R} |g(x)| \sum_{i=1}^2M_{\Phi}\tilde{u}(A_ix)dx
\\ & \leq \frac{c}{\lambda}\int_{\R} f(x) \sum_{i=1}^2M_{\Phi}\tilde{u}(A_ix)dx.
\end{align*}

To show $(b)$, we only have to estimate $III$. Using Theorem \ref{fuerte2pesos},
\begin{align*}
III&=u\{x \in  \R\setminus \tilde{\Omega}: |Tg(x)|>\lambda/2\}\leq \frac{2^p}{\lambda^p}\int_{\R} |Tg(x)|^p \tilde{u}(x)dx
\\&\leq \frac{c}{\lambda^p}\int_{\R}g(x)^p \sum_{i=1}^m M_{\D}u(A_ix)dx
\\ & \leq \frac{c}{\lambda}\int_{\R}f(x) \sum_{i=1}^m M_{\D}u(A_ix)dx.
\end{align*}

\end{proof}

\bibliographystyle{acm}
\bibliography{Biblio}

\end{document}